\documentclass[reqno,12pt]{amsart}
\usepackage[letterpaper,margin=1in,footskip=0.25in]{geometry}
\usepackage{amssymb,amsthm,amsmath,amscd}

\newtheorem{letterthm}{Theorem}                   
    
\theoremstyle{plain}               
\newtheorem{theorem}{Theorem}[section]       
\newtheorem{lemma}[theorem]{Lemma}          
\newtheorem{proposition}[theorem]{Proposition}
\newtheorem{corollary}[theorem]{Corollary}
\theoremstyle{definition}           
\newtheorem{definition}[theorem]{Definition}
\newtheorem{example}[theorem]{Example}
\newtheorem{remark}[theorem]{Remark}

\usepackage{mathtools}
\usepackage{soul}
\RequirePackage[dvipsnames,usenames]{color}
\usepackage{changepage}
\usepackage{appendix}
\usepackage{tikz}
\usepackage{tikz-cd}
\usetikzlibrary{cd}
\usepackage{graphicx}
\usepackage[all,cmtip]{xy}
\usepackage{todonotes}
\usepackage{bm}
\usepackage{pifont}
\usepackage{upgreek}

\usepackage{eucal}
\usepackage{ulem}
\usepackage{stmaryrd}
\usepackage{fullpage}
\usepackage{setspace}
\usepackage{enumerate}
\usepackage{calc}
\usepackage{bbding}
\usepackage{verbatim}
\usepackage{alltt}
\usepackage{enumitem}

\newcommand{\NN}{\mathbb{N}}

\DeclareMathOperator{\PP}{\mathcal{P}}
\DeclareMathOperator{\QQ}{\mathbb{Q}}

\newcommand{\fm}{\mathfrak{m}}

\DeclareMathOperator{\Hom}{Hom}
\DeclareMathOperator{\coker}{coker}

\usepackage[colorlinks=true, linkcolor=green, citecolor=blue, urlcolor=magenta]{hyperref}

\usepackage[nameinlink]{cleveref}

\crefname{theorem}{Theorem}{Theorems}
\crefname{lemma}{Lemma}{Lemmas}
\crefname{proposition}{Proposition}{Propositions}
\crefname{corollary}{Corollary}{Corollaries}
\crefname{definition}{Definition}{Definitions}
\crefname{example}{Example}{Examples}
\crefname{remark}{Remark}{Remarks}

\Crefname{theorem}{Theorem}{Theorems}
\Crefname{lemma}{Lemma}{Lemmas}
\Crefname{proposition}{Proposition}{Propositions}
\Crefname{corollary}{Corollary}{Corollaries}
\Crefname{definition}{Definition}{Definitions}
\Crefname{example}{Example}{Examples}
\Crefname{remark}{Remark}{Remarks}

\crefname{letterthm}{Theorem}{Theorems}
\Crefname{letterthm}{Theorem}{Theorems}

\usepackage{tikz-cd}
\usepackage{blindtext}
\usepackage{caption}
\usepackage{subcaption}
\usepackage{comment}
\usepackage{xfrac}

\usepackage[
    backend=biber,
    style=alphabetic,
    sorting=nyt,
    backref=true,
    doi=true,
    url=true
]{biblatex}
\DefineBibliographyStrings{english}{
  in = {}
}

\DeclareFieldFormat{mrnumber}{MR#1}

\AtEveryBibitem{\clearfield{note}}

\renewbibmacro*{doi+eprint+url}{%
  \printfield{doi}%
  \iffieldundef{doi}{%
    \printfield{url}%
  }{}%
  \setunit{\adddot\space}%
  \printfield{mrnumber}%
}
\renewbibmacro*{pageref}{%
  \iflistundef{pageref}
    {}
    {\setunit{\addcomma\space}%
     \printlist[pageref][-\value{listtotal}]{pageref}}}

\graphicspath{Images/}

\title{On $\fm$-adic Continuity of $F$-Splitting Ratio}
\author[Akter]{Maria Akter}
\address{Department of Mathematics, University of Alabama, Tuscaloosa, AL 35487 USA}
\email{makter2@crimson.ua.edu}
\urladdr{\url{https://mariaakter-math.github.io/}}
\begin{document}

\maketitle

\textbf{Abstract.} We investigate the $\mathfrak{m}$-adic continuity of Frobenius splitting dimensions and ratios for divisor pairs $(R,\Delta)$ in an $F$-finite local ring $(R,\mathfrak{m},k)$ of prime characteristic $p>0$. Our main result states that if $R$ is an $F$-finite, $\mathbb{Q}$-Gorenstein, Cohen-Macaulay local ring of prime characteristic $p>0$, the Frobenius splitting numbers $a^{\Delta}_e(R)$ remain unchanged under a suitable small perturbation.  Moreover, we establish a desirable inequality of Frobenius splitting dimensions under general perturbations. That is, $\dim (R/(\mathcal{P}(R/(f),\Delta|_{f})))\leq \dim (R/(\mathcal{P}(R/(f+\varepsilon),\Delta|_{(f+\varepsilon)})))$ for all $\varepsilon \in \fm^{N\gg0}$, providing an example that demonstrates strict improvement can occur.

\section{Introduction}

Let $(R,\fm,k)$ be a local $F$-finite ring of prime characteristic $p>0$. For each $e\in\NN$, let $F^e_{*}R$ be the finitely generated $R$-module obtained via the restriction of scalars under the $e$th
 iteration of the Frobenius map. If $e\in \NN$, consider the $e$th
 Frobenius splitting number of $R$, denoted by $a_e(R)$, defined by $F^e_{*}R\cong R^{\bigoplus{a_e(R)}}\bigoplus M_e$, where $M_e$ has no free $R$-summand. The Frobenius splitting ratio of $R$ is the asymptotic rate of the number $a_e(R)$ grows as $e\to \infty$, see \cite{AberbachEnescu2005}. If $(R,\fm,k)$ is a local $F$-finite, $\QQ$-Gorenstein, Cohen-Macaulay ring of prime characteristic $p>0$ and $0\neq f\in\fm$, then the Frobenius splitting numbers of $R/(f)$ can be measured through compatible Frobenius splittings in $R$ via techniques of Inversion of Adjunction for $F$-purity, see \cite{PolstraSimpsonTucker2025}. For other relevant situations, see \cite{Das2015,Fedder83,Schwede2009,Taylor2023}. This article employs this technique, which allows comparisons between the Frobenius splitting ratio of the pair $R/(f)$ and $R/(f+\varepsilon)$ for $\varepsilon \in \fm^N$ and $N\gg0$. As a motivating example let $R=k[[x_1,x_2,\dots,x_n]]$ be the power series ring in $n$ variables over a field $k$. If we add a term of sufficiently high order to an element $f\in R$, will the Frobenius splitting ratio be similar to that of $f$?

In 1956, Samuel \cite{Samuel1956} first introduced perturbation problems in local rings, showing that if $f$ defines a hypersurface with an isolated singularity, then for sufficiently small perturbations $g$, the quotient rings $R/(f)$ and $R/(g)$ remain isomorphic. Hironaka \cite{MR201433} extended this result to reduced, equidimensional isolated singularities, establishing stability under high-order perturbations. Cutkosky and Srinivasan \cite{CutkoskySrinivasan1993} further generalized these ideas to complete intersection prime ideals and reduced equidimensional ideals. Srinivas and Trivedi \cite{SrinivasTrivedi1996} demonstrated that in generalized Cohen-Macaulay rings, the associated graded rings of a parameter ideal remain unchanged under small perturbations.

Numerical invariants in positive characteristic have been widely studied, especially in relation to singularities under $\fm$-adic perturbations. Srinivas and Trivedi \cite{SrinivasTrivedi1996} showed that the Hilbert function generally decreases under small perturbations, while Ma, Quy, and Smirnov \cite{MaQuySmirnov2020} extended this by proving that the Hilbert function remains invariant for filter-regular sequences. Polstra and Smirnov \cite{PolstraSmirnov2020} demonstrate $\fm$-adic continuity of the Hilbert-Kunz multiplicity in parameter ideals of Cohen-Macaulay rings, and similar results are shown for the $F$-signature in Gorenstein rings \cite{PolstraSmirnov2020} and for $\QQ$-Gorenstein rings \cite{DeStefaniSmirnov2020}.
This area of research continues to develop in more recent works \cite{DeStefaniSmirnov2020,PolstraSimpson23}. 

In \cite[Theorem~3.6]{PolstraSimpson23}, Polstra and Simpson proved that if $(R, \fm, k)$ is a local $F$-finite, Cohen-Macaulay, $\QQ$-Gorenstein ring of prime
characteristic $p>0$ and $f \in \fm$ is a nonzero divisor such that $R/(f)$ is Gorenstein in codimension~1 and $F$-pure, then there exists $N\gg0$ such that $R/(f + \varepsilon)$ is $F$-pure for every $\varepsilon \in \fm^N$. In \cite[Theorem~3.8, Corollary~3.9]{DeStefaniSmirnov2020}, De Stefani and Smirnov prove the $\fm$-adic continuity of another type of numerical invariant, defined by Hochster and Yao in \cite[Definition~0.1]{YaoHochster2022}, called the $F$-rational signature $s_{\mathrm{rat}}(R)$ for an $F$-finite Cohen-Macaulay local ring of prime characteristic $p>0$. 

A natural question that arises is whether $\mathfrak{m}$-adic continuity holds for another type of numerical invariant used to measure singularities, known as the Frobenius splitting ratio, denoted by $r_F(R)$. This invariant was first introduced by Aberbach and Enescu in \cite[Definition~2.4]{AberbachEnescu2005}. Specifically, let $(R,\fm,k)$ be a local $F$-finite, $F$-pure ring of prime
characteristic $p>0$, where $k$ is a perfect field, and set $n\coloneqq\dim\bigl(R/\PP(R)\bigr),$
where $\PP(R)$ denotes the splitting prime ideal of $R$. We refer to $n$ as
the splitting dimension of $R$. The Frobenius splitting ratio of $R$
is then defined by
$r_F(R)\coloneqq \lim_{e\to\infty}\frac{a_e}{p^{en}},$
where this limit is known to exist, see \cite{MR2980498}.
Moreover, it satisfies $0< r_F(R)\le 1.$ 
The focus of this article is to study splitting ratios and splitting dimensions of perturbations under suitable hypotheses; the following applications are of the main results of the article.
\begin{letterthm}
\label{thm:A}
   Let $(R,\fm, k)$ be a local $F$-finite, $F$-pure, $\QQ$-Gorenstein, Cohen-Macaulay ring of prime characteristic $p>0$ and $f\in R$ be a nonzero divisor of $R$. Let $\PP(R/(f))$ be the lift of the splitting prime of $R/(f)$ in $R$, then there exist an $\fm$-primary ideal $\mathfrak{a}$ in $R$ and $e_0\in \NN$ such that for all $\varepsilon\in \mathfrak{a}\cap\PP^{[p^{e_{0}}]} \ \text{and all } e\in\NN$
 \[
 a_e(R/(f))=a_e(R/(f+\varepsilon)). 
 \]
In particular, if $\PP(R/(f))=\fm/(f)$, then there exists an $\fm$-primary ideal $\mathfrak{a}$ so that for all $\varepsilon\in\mathfrak{a}$, \[
\PP(R/(f+\varepsilon))=\fm/(f+\varepsilon).
\] 
\end{letterthm}

Our next main result establishes a desirable relationship between the splitting dimensions of $R/(f)$ and $R/(f +\varepsilon)$ for $\varepsilon\in \fm^{[p^{e_0}]}$.
\begin{letterthm}
\label{thm:B}
 Let $(R,\fm,k)$ be a local $F$-finite, $F$-pure, $\QQ$-Gorenstein, Cohen-Macaulay ring of prime characteristic $p>0$ and $f\in \fm$ be a nonzero divisor of $R$. Then there exists an integer $e_{0}\in \NN$ such that for all $\varepsilon\in \fm^{[p^{e_0}]}$,
 \[
    \dim \bigl(R/\PP(R/(f))\bigr)\leq \dim\bigl (R/\PP(R/(f+\varepsilon))\bigr).
    \]
\end{letterthm}

However, equality in Theorem~\hyperref[thm:B]{\ref*{thm:B}} does not hold in general, as illustrated in Example~\hyperref[ex:four-two]{\ref*{ex:four-two}}.

Theorem~\hyperref[thm:C]{\ref*{thm:C}} is a result of independent interest and is applicable to any $F$-finite Gorenstein local ring.
\begin{letterthm}
\label{thm:C}       
Let $(R,\mathfrak{m},k)$ be an $F$-finite, $F$-pure, Gorenstein  local ring of dimension $d$ of prime characteristic $p>0$. The following are equivalent:
 \begin{enumerate}[label=(\roman*)]
 \item$a_1(R)=[k^{1/p}:k].$
 \item There exists $e\in \mathbb{N}$ such that $a_e(R)=[k^{1/p^e}:k]$.
 \item For all $e\in \mathbb{N}$, $a_e(R)=[k^{1/p^e}:k]$.
     \item$I_1(R)=\fm.$
     \item There exists $e\in \mathbb{N}$ such that $I_e(R)=\fm$. 
     \item For all $e\in \mathbb{N}$, $I_e(R)=\fm$.
     \item The splitting prime ideal $\mathcal{P}(R)$ of $R$ equals $\fm$.
 \end{enumerate}
\end{letterthm}
\section*{Acknowledgments}
 This work is based on the author’s thesis. The author would like to thank her advisor, \textbf{Thomas Polstra}, for being so generous with his time, constant support, guidance, and encouragement throughout the research journey. His expertise and mentorship have 
\section{Preliminaries}\label{sec:prelem}

Throughout this paper, $(R,\fm,k)$ denotes a Noetherian local ring of prime characteristic $p>0$, with residue field $k=R/\fm$. The Frobenius endomorphism $F: R \to R$ will be denoted as $F$ and $F^e$ denotes the $e$th iteration of Frobenius endomorphism $F$. For a given ideal $I \subseteq R$ and $e \in \NN$, let $I^{[p^e]}\coloneqq(i^{p^e} \mid i \in I)$ represent the expansion of $I$ under $F^e$. We say $R$ is \emph{$F$-finite} if $F^e_{*}R$ is finitely generated $R$-module. If $M$ is an $R$-module then $F^e_{*}M$ is the $R$-module obtained by the restrictions of scalars under $F^e$.

The following classical result is fundamental for understanding the number of free direct $R$-summands of $F^e_* R$.

\begin{theorem}\label{thm:2.1}
\textup{\cite{Hochster77}} 
Let $(R,\fm,k)$ be an excellent local ring, and let $M$ be a finitely generated $R$-module equipped with a map $R \to M$. Then the following statements are equivalent:
\begin{enumerate}[label=(\roman*)]
    \item The map $R \to M$ splits.
    \item The induced map $E_R(k) \to E_R(k) \otimes_R M$ is injective, where $E_R(k)$ denotes the injective hull of $k$.
    \item The map $R \to M$ is pure; that is, for any $R$-module $N$, the natural map $N \to M \otimes_R N$ is injective.
\end{enumerate}
\end{theorem}

Let $M$ be a finitely generated $R$-module that has no free direct $R$-summands. Then for every $m \in M$, the map 
\[
R \longrightarrow M, \quad 1 \mapsto m,
\] 
does not split. Therefore, by Theorem~\hyperref[thm:2.1]{\ref*{thm:2.1}}, the induced map
\[
E_R(k) \longrightarrow E_R(k) \otimes_R M, \quad a \mapsto a \otimes m,
\] 
is not injective. The following results from \cite{AberbachEnescu2005} are fundamental in defining the splitting prime ideal of $R$.

\begin{proposition}\textnormal{\cite{AberbachEnescu2005}}
    Let $(R, \fm, k)$ be a local ring, $M$ be a finitely generated $R$-module, and let $u$ be the socle generator of $E_R(k)$. Then $M$ has no nonzero free direct summands if and only if for every $m \in M$, 
    \[
u \otimes m = 0 \quad \text{in } E_R(k) \otimes_R M.
\]
\end{proposition}
\begin{corollary}\textnormal{\cite{AberbachEnescu2005}}
    Let $(R, \fm, k)$ be an $F$-finite, reduced local ring of characteristic $p$. Let 
$F^e_{*}R = R^{a_{e}(R)} \oplus M_{e}$ be a direct sum decomposition of $F^e_{*}R$ over $R$ as an $R$-module, where $M_{e}$ has no free direct $R$-summands. Define
\[
I_e(R) := \bigl\{ r \in R \mid F^e_* r \otimes u = 0 \ \text{in } F^e_* R \otimes_R E_R(k) \bigr\}.
\]  
Then, $a_{e}(R) = \lambda_R\bigl(F^e_{*}R/F^{e}_{*}I_{e}(R)\bigr)$, where $\lambda_R(-)$ denotes the length as an $R$-module.
\end{corollary}

The concept of splitting prime ideal, Frobenius splitting dimension, and Frobenius splitting ratio were first introduced by Aberbach and Enescu \cite{AberbachEnescu2005}. The following definition is from \cite{AberbachEnescu2005}.
\begin{definition}\cite[Definition~3.2]{AberbachEnescu2005}
Let $(R,\fm,k)$ be an $F$-finite, $F$-pure local ring of prime characteristic $p>0$. For every $c\in R$ and $e\in \NN $, let $\phi_{c,e}:R\to F^e_{*}R$ be the $R$-linear map defined by $\phi_{c,e}(1)=F^e_{*}c$. Then the \emph{splitting prime ideal} of $R$ is defined by
        \begin{align*}
     \PP(R)&= \{ c\in R \mid \phi_{c,e} \, \text {does not split for any }e\in \NN\}\\&=\bigcap\limits_{e\in \NN} I_e(R).
 \end{align*}
\end{definition}  

\begin{remark}
    For a local ring $(R,\fm,k)$ and $0\neq f\in\fm$, we identify $\PP(R/(f))$ as the lift of the splitting prime ideal $\PP(R)$ in $R$.
\end{remark}

\begin{definition}
    Let $(R,\fm,k)$ be a local $F$-finite ring of prime  characteristic $p>0$. Let $d=\dim (R)$, then the \emph{$F$-signature} of $R$ is \[s(R)=\lim_{e \to \infty }\frac {\lambda_{R}(R/I_{e}(R))}{p^{ed}}.\]
    This limit always exists by \cite{Tucker2012} and enjoys the property
        \[0\leq s(R)\leq 1.\]
      \end{definition}  
    
The $F$-signature of a local ring $(R,\fm,k)$ detects regularity, i.e., $s(R)=1$ if and only if $R$ is a regular local ring \cite[Corollary~16]{Huneke_2002}. Also, $F$-signature detects strong $F$-regularity, i.e., $s(R)>0$ if and only if $R$ is strongly $F$-regular \cite[Main Result]{AberbachLeuschke2003}.

\begin{definition}
    Let $(R,\fm,k)$ be a local $F$-finite, $F$-pure ring of prime characteristic $p>0$. Let $n\coloneqq\dim (R/\PP(R))$, we call $n$ the \emph{splitting dimension} of $R$. The \emph{Frobenius splitting ratio} of $R$ is
    \[
    r_{F}(R)\coloneqq\lim_{e\to \infty}\frac{\lambda_{R}(R/I_{e}(R))}{p^{en}}.
    \]
        This limit always exists by \cite[Corollary~4.3]{MR2980498} and enjoys the property
        \[0<r_F(R)\leq 1.\]
\end{definition}

The Frobenius splitting ratio of $R$ is a generalization of $F$-signature and accurately measures the asymptotic behavior of the Frobenius splitting numbers of $R$ if $\PP(R)\neq {0}$. The results of \cite{MR2980498} on $F$-splitting ratios are generalization of results of \cite{Tucker2012} and \cite{AberbachEnescu2005} on $F$-signature.
\begin{remark}
Let $(R,\fm,k)$ be a local ring and let $0\neq f \in \fm$. It is natural to expect that for sufficiently small perturbations $\varepsilon \in \fm^N$ with $N \gg 0$, the quotient rings $R/(f)$ and $R/(f+\varepsilon)$ share similar properties. In particular, if $R$ is $\QQ$-Gorenstein and of prime characteristic $p>0$, then for every $\delta>0$ there exists $N$ such that for all $\varepsilon \in \fm^N$,
\[
\left| s(R/(f)) - s(R/(f+\varepsilon)) \right| < \delta,
\]
see \cite{DeStefaniSmirnov2020,MaQuySmirnov2020}. This result illustrates the $\fm$-adic continuity of the $F$-signature under small perturbations of the defining parameter.
\end{remark}
We briefly review the definition of $\QQ$-Gorenstein ring and $F$-purity of a pair $(R,\Delta)$.
\begin{definition}Let $(R, \mathfrak{m}, k)$ be a Noetherian local ring that is $(S_2)$ and $(G_1)$. A \emph{Weil divisor $D$} is an (integral) divisor of $\operatorname{Spec}(R)$ such that $R(D)_\mathfrak{p} \cong R_\mathfrak{p}$ for all height~1 prime ideals $\mathfrak{p} \in \operatorname{Spec}(R)$. If $D$ is a Weil divisor of $\operatorname{Spec}(R)$, then $R(D)$ denotes the corresponding \emph{reflexive} fractional ideal of $R$. A divisor $D$ is called \emph{principal} if $R(D) \cong R$. Equivalently, there exists a nonzero divisor $a$ in $R$ such that $D = \operatorname{div}(a)$ and 
$R(D) = R(\operatorname{div}(a)) = R \cdot \frac{1}{a}.$
We say that a Weil divisor $D$ is \emph{effective}, and write $D \geq 0$ if $R \subseteq R(D)$. A $\mathbb{Q}$-divisor $\Delta = \sum a_i D_i$ is a finite $\mathbb{Q}$-linear combination of Weil divisors $D_i$ of $\operatorname{Spec}(R)$. A $\mathbb{Q}$-divisor is called \emph{$\mathbb{Q}$-Cartier} if there exists a natural number $n>0$ such that $n\Delta$ is principal. The \emph{index} of a $\mathbb{Q}$-Cartier $\mathbb{Q}$-divisor $\Delta$, is the smallest positive integer $n$ such that $n\Delta$ is principal. A \emph{canonical divisor} of $\operatorname{Spec}(R)$ is a Weil divisor $K_R$ such that $R(K_R)$ is a canonical module of $R$. We say that $R$ is $\mathbb{Q}$-Gorenstein if $K_R$ is a $\mathbb{Q}$-Cartier divisor.

Suppose that $f \in R$ is a nonzero divisor so that $R/(f)$ enjoys the property of being $(S_2)$ and $(G_1)$. Let $D$ be a Weil divisor of $\operatorname{Spec}(R)$ with components disjoint from those of $\operatorname{div}(f)$. If $(R(D)/fR(D))_\mathfrak{p} \cong (R/(f))_\mathfrak{p}$ for all height~1 prime ideals $\mathfrak{p}$ of $R/(f)$, then we write $D|_{R/(f)}$ to be the restriction of $D$ to $R/(f)$. In particular, if $( - )^{**}$ denotes reflexification with respect to $R/(f)$, then
\[
    \left( \frac{R(D)}{fR(D)} \right)^{**} \cong (R/(f))(D|_{R/(f)}).
\]
If $\Delta$ is a $\mathbb{Q}$-divisor, $n \neq 0$ so that $n\Delta$ is an integral Weil divisor on $R$, and $(n\Delta)|_{R/(f)}$ is a Weil divisor of $R/(f)$, then we let $\Delta|_{R/(f)} = \frac{1}{n} (n\Delta)|_{R/(f)}.$    
\end{definition}

\begin{definition}Let $(R, \mathfrak{m}, k)$ be an $F$-finite local ring of prime characteristic $p>0$ that is $(S_2)$ and $(G_1)$, and let $\Delta \geq 0$ be an effective $\mathbb{Q}$-divisor. The \emph{$e$th splitting ideal} of the pair $(R, \Delta)$ is the ideal
\[
    I_e(R, \Delta) = \{x \in R \mid R\xrightarrow{\cdot F^e_* x} F^e_* R(\lceil(p^e - 1) \Delta\rceil) \text{ does not split}\}.
\]
We say that the pair $(R, \Delta)$ is $F$-pure if there exists $e \in \mathbb{N}$ so that $I_e(R, \Delta) \neq R$.
\end{definition}
\begin{remark}
    If $(p^e-1)\Delta$ is integral, then $(p^e-1)\Delta=\lceil(p^e - 1) \Delta\rceil$.
\end{remark}

Fedder's Criterion describes an explicit criterion for a hypersurface to be $F$-pure, and the methods of associated proof are a tool for describing the splitting ideals of a hypersurface as quotients of ideals of the ambient regular ring. The proof of the Inversion of Adjunction of $F$-purity theorem \cite[Theorem~A]{PolstraSimpsonTucker2025} also provides an explicit description of the splitting ideals of a pair restricted to a hypersurface under minimal hypotheses.
\begin{remark}\label{rem:2.12}
    Let $(R,\fm,k)$ be a $d$-dimensional local ring and $M$ be a finitely generated $R$-module. Then
    \begin{enumerate}[label=(\roman*)]
        \item If $\eta:M\to M^{**}$ is an isomorphism in codimension 1, where
        \[ M^{**}\coloneqq\Hom_{R}\Bigl(\Hom_{R}(M,R),R\Bigr),\] then $H^d_{\fm}(M)\cong H^d_{\fm}(M^{**})$.
        Indeed, let $K = \ker(\eta)$ and $C = \coker(\eta)$. Then we have an exact sequence
\[
0 \longrightarrow K \longrightarrow M \xrightarrow{\eta} M^{**} \longrightarrow C \longrightarrow 0.
\]
which give rise to the short exact sequences:
\begin{equation}
0 \longrightarrow K \longrightarrow M \longrightarrow L  \longrightarrow 0.
\tag{2.12.1}
    \label{eq:2.12.1}
\end{equation}
and
\begin{equation}
0 \longrightarrow L \longrightarrow M^{**} \longrightarrow C  \longrightarrow 0.
\tag{2.12.2}
    \label{eq:2.12.2}
\end{equation}
From (\ref{eq:2.12.1}), we have
\[
 H^d_{\mathfrak m}(K)\longrightarrow H^d_{\mathfrak m}(M)  \longrightarrow H^d_{\mathfrak m}(L) \longrightarrow 0.
\]
Since $\eta$ is an isomorphism in codimension $1$, we have $\dim K \le d-2$, and hence
$H^d_{\fm}(K)=0$. Therefore,
\[
H^d_{\fm}(M) \cong H^d_{\fm}(L).
\]
Similarly, from (\ref{eq:2.12.2}), we obtain
\[
 H^{d-1}_{\mathfrak m}(C)\longrightarrow H^d_{\mathfrak m}(L)  \longrightarrow H^d_{\mathfrak m}(M^{**}) \longrightarrow H^d_{\mathfrak m}(C) \longrightarrow 0.
\]
Again, since $\dim C \le d-2$, we have $H^{d-1}_{\fm}(C)=H^d_{\fm}(C)=0$, and hence
\[
H^d_{\fm}(L) \cong H^d_{\fm}(M^{**}).
\]
Thus,
\[
H^d_{\fm}(M) \cong H^d_{\fm}(M^{**}).
\]
        \item If $N$ is an $R$-module, then $M\otimes_{R}H^d_{\fm}(N)\cong H^d_{\fm}(M\otimes_{R}N).$ Indeed, top local cohomology is the cokernal of a \v{C}ech complex, and tensor product is right exact.
    \end{enumerate}
\end{remark}  
\begin{theorem}\label{thm:2.13}Let $(R,\mathfrak{m},k)$ be an $F$-finite, $(S_2)$ and $(G_1)$ local ring of prime characteristic $p>0$. Let $\Delta\geq0$ be a $\mathbb{Q}$-divisor on $\operatorname{Spec}(R)$ such that $\Delta + K_R$ is a $\mathbb{Q}$-Cartier divisor and $(p^e-1)\Delta$ is an integral Weil divisor for all $e \gg 0$ and divisible. Let $f \in R$ be a nonzero divisor of $R$ so that the components of $\Delta$ are disjoint from the components of $\operatorname{div}(f)$, $R/(f)$ is $(S_2)$ and $(G_1)$, and $\Delta|_{R/(f)}$ is a $\mathbb{Q}$-divisor of \ $\operatorname{Spec}(R/(f))$. Then there exists an $e_0 \in \mathbb{N}$ so that for all $t \geq 1$,
\[
I_{t e_0} \left( R/(f), \Delta|_{R/(f)} \right) = \frac{(I_{t e_0} (R, \Delta) :_R f^{p^{te_0}-1})}{(f)}.\]
\end{theorem}
\begin{proof}The pair $(R, \Delta + \operatorname{div}(f))$ is $F$-pure if and only if the pair $(R/(f), \Delta|_{R/(f)})$ is $F$-pure by \cite[Theorem~A]{PolstraSimpsonTucker2025}. Therefore,
$I_e(R/(f), \Delta|_{R/(f)}) = R/(f) \text{ for all } e \in \mathbb{N}$
if and only if
$I_e(R, \Delta + \operatorname{div}(f))=R \text{ for all } e \in \mathbb{N}.$ We therefore may assume without loss of generality that the pairs $(R, \Delta + \operatorname{div}(f))$ and $(R/(f), \Delta|_{R/(f)})$ are $F$-pure.

Our assumptions imply that there exists an $e_0 \in \mathbb{N}$ so that $(p^{e_0}-1)(\Delta + K_R)$ is $\mathbb{Q}$-Cartier of index $p^{e_1}$ for some $e_1\in \mathbb{N}$. Therefore, for every $t\in \mathbb{N}$,
\[(p^{te_0} -1)(\Delta + K_R)=\frac{p^{te_0}-1}{p^{e_0}-1}(p^{e_0}-1)(\Delta+K_R)\] is $\mathbb{Q}$-Cartier of index $p^{e_1}.$\\
To ease notation, for each $e\in \mathbb{N}$, let
\[
\Delta_e \coloneqq (p^e - 1)(\Delta + K_R) \quad \text{and} \quad 
\bar{\Delta}_e \coloneqq (p^e - 1)(\Delta|_{R/(f)} + K_{R/(f)}).
\] 
By \cite[Lemma~4.1]{PolstraSimpsonTucker2025}, we have
\begin{equation}
    \frac{R((p^{te_0}-1)\Delta)}{fR((p^{te_0}-1)\Delta)}\cong R/fR(((p^{te_0}-1)\Bar{\Delta})).
    \tag{2.13.1}
    \label{eq:2.13.1}
\end{equation}

Moreover, we have
\begin{equation}
  H^{d-1}_\fm(R/fR(\Bar{\Delta}_{te_0}+K_R))=\text{Ker}\left( H^d_{\fm}(R(\Delta_{te_0}+K_R))\xrightarrow{\cdot f} H^d_{\fm}(R(\Delta_{te_0}+K_R))\right).
  \tag{2.13.2}
  \label{eq:2.13.2}
\end{equation}

Let $g\in I_{t e_0}\left( R, \Delta\right)$, this implies the map
$R\xrightarrow{\cdot F^{te_0}_{*} g} F^{te_0}_* R((p^{te_0} - 1)\Delta)$ does not split. That is, $H^d_\fm(R(K_R)) \xrightarrow{\cdot F^{te_0}_{*}g} F^{te_0}_{*} H^d_\fm (R(\Delta_{te_0}+K_R))$ is not injective. To prove that $g\in I_{t e_0} \left( R/(f), \Delta|_{R/(f)} \right)$, consider the diagram which is commutative by (\ref{eq:2.13.1}),
\begin{equation}
\begin{tikzcd}
0 \arrow[r]  & R \arrow[r, "\cdot f"] \arrow[d, "\cdot F^{te_{0}}_{*}f^{p^{t{e_0}}-1}g"] & R \arrow[r] \arrow[d, " F^{te_0}_{*}g"] & R/(f)\arrow[r] \arrow[d, "F^{te_0}_{*}g"]  & 0  \\
0 \arrow[r] &F^{te_0}_{*} R(\Delta_{te_0}) \arrow[r, "\cdot F^{te_0}_{*}f"] & F^{te_0}_{*}R(\Delta_{te_0}) \arrow[r]& F^{te_0}_{*}R/(f)(\Bar{\Delta}_{te_0}) \arrow[r] & 0 
\end{tikzcd}
\tag{2.13.3}
\label{eq:2.13.3}
\end{equation}
which induces the following commutative diagram by (\ref{eq:2.13.2}),
{\scriptsize\begin{equation}
\begin{tikzcd}[row sep=large, column sep=3.5mm]
\quad 0 \arrow[r]& H^{d-1}_\fm (R/(f)(K_{R/(f)})) \arrow[r] \arrow[d, "F^{te_{0}}_{*}g"]& H^{d}_\fm(R(K_R)) \arrow[r, "\cdot f"] \arrow[d, "\cdot F^{te_{0}}_{*}f^{p^{t{e_0}}-1}g"]& H^{d}_\fm(R(K_R))\arrow[r] \arrow[d, "F^{te_0}_{*}g"]& 0 \\0 \arrow[r]&F^{te_0}_{*}H^{d-1}_{\fm} (R/(f)(\Bar{\Delta}_{te_0}+K_{R/(f)})) \arrow[r]& F^{te_0}_{*}H^d_{\fm}(R(\Delta_{te_0}+K_R)) \arrow[r,"\cdot F^{te_0}_*f"]& F^{te_0}_{*}H^d_{\fm}(R({\Delta}_{te_0}+K_R)) \arrow[r]& 0
\end{tikzcd}
\tag{2.13.4}
\label{eq:2.13.4}
\end{equation}
}
The local cohomology modules $H^{d-1}_{\mathfrak{m}}(R/(f)(K_{R/(f)}))$ and $H^{d}_{\mathfrak{m}}(R(K_R))$ are essential extensions of the residue field.
Therefore, the left vertical map is injective if and only if the middle vertical map is injective because of the following isomorphism of the vertical maps in (\ref{eq:2.13.4}) by Remark~\hyperref[rem:2.12]{\ref*{rem:2.12}}:

\begin{align*}
\Big(H^{d-1}_{\mathfrak{m}}(R/(f)(K_{R/(f)}))\xrightarrow{\cdot F^{te_0}_{*}g} & F^{te_0}_{*}H^{d-1}_{\mathfrak{m}}(R/(f)(\Bar{\Delta}_{te_0}+K_{R/(f)}))\Big)\\ & \cong
\Big(R/(f)\xrightarrow{\cdot F^{te_0}_{*}g} F^{te_0}_{*}R/(f)(\Bar{\Delta}_{te_0})\Big)\otimes_{R/(f)}H^{d-1}_{\mathfrak{m}}(R/(f)(K_{R/(f)}))
\end{align*}
and 
\begin{align*}
\left(H^{d}_{\mathfrak{m}}(R(K_R))\xrightarrow{\cdot F^{te_0}_{*}f^{p^{te_0}-1}g}F^{te_0}_{*}H^{d}_{\mathfrak{m}}(R(\Delta_{te_0}+K_R))\right)\\& \cong \left(R\xrightarrow{\cdot F^{te_0}_{*}f^{p^{te_0}-1}g}F^{te_0}_{*}R(\Delta_{te_0})\right)\otimes _{R} H^{d}_{\mathfrak{m}}(R(K_R)).
\end{align*}
Suppose $0 \neq g\in I_{te_0}(R/(f),\Delta|_{R/(f)}))$, then the map \[
      R/(f) \xrightarrow{\cdot F^{te_0}_{*}g} F^{te_0}_{*} R/(f)((p^{te_0}-1)\Delta|_{R/(f)}) \, \text{does not split.}
      \]   
      This is equivalent to saying that the map \[
      R \xrightarrow{\cdot F^{te_0}_{*}f^{p^{te_0}-1}g } F^{te_0}_{* }R((p^{te_0}-1)\Delta) \text{ does not split}\]
      by diagram (\ref{eq:2.13.3}) and (\ref{eq:2.13.4}). This condition holds true if and only if $gf^{p^{te_0}-1}\in I_{t e_0}\left(R,\Delta\right)$, which is true if and only if $g\in (I_{te_0}(R,\Delta):_Rf^{p^{te_0}-1})$, for all $t\in\mathbb{N}$.
 \\Thus, 
 \[
    I_{t e_0} \left( R/(f), \Delta|_{R/(f)} \right) = \frac{(I_{t e_0} (R,\Delta):_R f^{p^{te_0}-1})}{(f)}.
\]
\end{proof}
\begin{lemma}\label{lemma:2.14} Let $(R,\mathfrak{m},k)$ be a local Cohen-Macaulay, $F$-finite, $F$-pure, $(S_2)$ and $(G_1)$ ring of prime characteristic $p>0$, and let $K_R$ be a canonical divisor of $\operatorname{Spec}(R)$. Let $E$ be a $p$-power torsion Weil divisor of $\operatorname{Spec}(R)$. Then $R(E + K_R)$ is a Cohen-Macaulay module.
\end{lemma}
\begin{proof}Assume that $p^eE \sim 0$, then $-p^eE\sim0.$
   Since $R$ is $F$-pure, the Frobenius map 
   \[R\to F^e_{*}R\] splits as a $R$-linear map. Thus, 
   \[F^e_{*}R\cong R\oplus\cdots.\]
Tensoring with $R(- E)$ gives,
\begin{equation}
F^e_{*}R\otimes_{R}R(-E)\cong R\otimes_{R}R(-E)\oplus\cdots\cong R(-E)\oplus\cdots
\tag{2.14.1}
\label{eq:2.14.1}
\end{equation}
Let $(-)^{*}=\Hom_{R}(-,R)$, then $(F^{e}_{*}R\otimes R(-E))^{**}\cong F^e_{*}R(-p^eE)$, and $R(-E)^{**}\cong R(-E)$.
But, $F^e_{*}R\otimes_{R}R(-E)\cong F^e_{*}R(-p^eE)\cong F^e_{*}R$. Thus, \[F^e_{*}R\cong R(-E)\oplus\cdots.\] 

Now since $R(-E)$ is a direct summand of Cohen-Macaulay module $F^e_{*}R$, so it is Cohen-Macaulay. Applying $\Hom_{R}(- ,R(K_R))$ in (\ref{eq:2.14.1}) gives,
\begin{align*}
\Hom_{R}(F^e_{*}R, R(K_R)) &\cong \Hom_{R}(R(-E),R(K_R)) \oplus \cdots \\
&\cong \Hom_{R}(R,R(E+K_R)) \oplus \cdots \\
&\cong R(E+K_R) \oplus \cdots
\end{align*}

But, $\Hom_{R}(F^e_{*}R, R(K_R))\cong F^e_{*}R(K_R)$, thus, $F^e_{*}R(K_R)\cong R(E+K_R)\oplus\cdots$. Therefore, $R(E+K_R)$ is a direct summand of Cohen-Macaulay module $F^e_{*}R(K_R)$, thus $R(E+K_R)$ is Cohen-Macaulay.  
\end{proof}
\begin{proposition}\label{proposition:2.15}
    Let $(R, \mathfrak{m}, k)$ be a local Cohen-Macaulay, $F$-finite, $F$-pure, $(S_2)$ and $(G_1)$ ring of prime characteristic $p>0$, and let $K_R$ be a canonical divisor of \ $\operatorname{Spec}(R)$. Let $\Delta \geq 0$ be an effective $\mathbb{Q}$-divisor so that $\Delta + K_R$ is $\mathbb{Q}$-Cartier and for all $e\gg0$ and divisible, $(p^e-1)\Delta$ is integral. Let $(\underline{x})$ be a system of parameters of $R$ and $u\in R(K_R)$ a socle generator modulo $(\underline{x})$ such that 
$((\underline{x}) R(K_R) :_{R(K_R)} \mathfrak{m}) = (\underline{x}) R(K_R) + Ru.$
For each $e \in \mathbb{N}$, let $\Delta_e = (p^e - 1)(\Delta + K_R)$. Then there exists an $e_0 \in \mathbb{N}$ such that for all $t\geq 1$,
\[I_{t e_0}(R,\Delta)=\left((\underline{x})^{[p^{t e_0}]} R(\Delta_{t e_0} + K_R) :_{R} u^{p^{t e_0}}\right).\]
\end{proposition}
\begin{proof} There exist $e_0,e_1$ so that for all $t
\in \NN$, $(p^{te_0}-1)(\Delta+K_R)$ is integral and $p^{e_1}(p^{te_0}-1)(\Delta+K_R)\sim 0.$ Let $g\in I_{t e_0}(R,\Delta)$, this implies that the map, 
\[R\xrightarrow{ \cdot F^{te_0}_* g} F^{te_0}_* R((p^{te_0} - 1)\Delta_{te_0})\] does not split. Equivalently, the map
\[H^d_{\mathfrak{m}} (R(K_R)) \to F^{te_0}_*(R (p^{te_0}-1) \Delta_{te_0}) \otimes_{R} H^d_\fm (R(K_R))\] is not injective. By Remark~\hyperref[rem:2.12]{\ref*{rem:2.12}},
\begin{align*}
  F^{te_0}_*(R (p^{te_0}-1) \Delta_{te_0}) \otimes_{R} H^d_\fm (R(K_R)) & \cong H^d_{\fm}(F^{te_0}_{*}R((p^{te_0}-1)\Delta_{te_0})\otimes_{R} R(K_R))\\ & \cong  H^d_{\fm}(F^{te_0}_{*}R((p^{te_0}-1)\Delta_{te_0}+p^{te_0}K_R))\\ & \cong H^d_{\fm}(F^{te_0}_{*}R(\Delta_{te_0}+K_R))\\ & \cong F^{te_0}_{*}H^d_{\fm}(R(\Delta_{te_0}+K_R)).
\end{align*}
Also,
\[
H^d_\fm (R(K_R))\cong \varinjlim_s \left(\frac{R(K_R)}{(\underline{x}^s)R(K_R)}\xrightarrow{\cdot x_1x_2\cdots x_d}\frac{R(K_R)}{(\underline {x}^{s+1})R(K_R)}\right),
\]
{\footnotesize
\[
F^{te_0}_{*}H^d_{\fm}(R(\Delta_{te_0}+K_R)) \cong F^{te_0}_{*} \varinjlim_s\left(\frac{R(\Delta_{te_0}+K_R)}{(\underline{x}^{sp^{te_0}})R(\Delta_{te_0}+K_R)}\xrightarrow{\cdot (x_1x_2\cdots x_d)^{p^{te_0}}}\frac{R(\Delta_{te_0}+K_R)}{(\underline {x}^{(s+1)p^{te_0}})R(\Delta_{te_0}+K_R)}\right).
\]
}The maps above in the direct limit system are injective by Lemma~\hyperref[lemma:2.14]{\ref*{lemma:2.14}}.\\
Therefore, the map, 
\[H^d_{\mathfrak{m}} (R(K_R)) \to F^{te_0}_*(R (p^{te_0}-1) \Delta) \otimes_{R} H^d_\fm (R(K_R))\]is not injective if and only if
\[[u+(\underline{x})R(K_R)]\mapsto [gu^{p^{te_0}}+(\underline{x}^{p^{te_0}})R(\Delta_{te_0}+K_R)]=0.\] Thus, $gu^{p^{te_0}} \in (\underline{x}^{p^{te_0}})R(\Delta_{te_0}+K_R)$. Hence, $g \in ((\underline{x})^{[p^{t e_0}]} R(\Delta_{t e_0} + K_R) :_{R} u^{p^{t e_0}}).$
\end{proof}
\begin{remark}If $(R,\fm,k)$ is Gorenstein and $\Delta=0$, then Proposition~\hyperref[proposition:2.15]{\ref*{proposition:2.15}} verifies 
\[I_{e}(R) = \left((\underline{x})^{[p^{e}]} :_{R} u^{p^{e}} \right).\]
See also
\cite[Theorem~11(2), Lemma~12]{Huneke_2002}.
\end{remark}
\section{\texorpdfstring{On $\fm$-adic Continuity of $F$-Splitting Ratio}{On m-adic Continuity of F-Splitting Ratio}}
\label{sec:theorem}
Theorem~\hyperref[theorem:3.1]{\ref*{theorem:3.1}} applies to all $F$-finite Gorenstein local ring of prime characteristics $p>0$. Intrinsically, the theorem plays a key role in the proofs of our main results.
\begin{theorem}\label{theorem:3.1}       
Let $(R,\mathfrak{m},k)$ be an $F$-finite, $F$-pure, Gorenstein  local ring of dimension $d$ of prime characteristic $p>0$. The following are equivalent:
 \begin{enumerate}[label=(\roman*)]
 \item$a_1(R)=[k^{1/p}:k].$
 \item There exists $e\in \mathbb{N}$ such that $a_e(R)=[k^{1/p^e}:k]$.
 \item For all $e\in \mathbb{N}$, $a_e(R)=[k^{1/p^e}:k]$.
     \item $I_1(R)=\fm.$
     \item There exist $e\in \mathbb{N}$ such that $I_e(R)=\fm$. 
     \item For all $e\in \mathbb{N}$, $I_e(R)=\fm$.
     \item The splitting prime ideal $\mathcal{P}(R)$ of $R$ equals $\fm$.
 \end{enumerate}
\end{theorem}
\begin{proof}We first establish the relationship between the splitting numbers $a_e(R)$ and the ideals $I_e(R)$. If $a_e(R)=[k^{1/p^e}:k],$ then $a_e(R)\geq1$ and $I_e(R)\subseteq \fm.$ 
 Recall that
\[
a_e(R)=\dim_k\!\left(\frac{F^e_*R}{F^e_*I_e(R)}\right).
\]
Since $I_e(R)\subseteq \fm$, we have $F^e_*I_e(R)\subseteq F^e_*\fm$, and hence
\[
\dim_k\!\left(\frac{F^e_*R}{F^e_*I_e(R)}\right)
\ge
\dim_k\!\left(\frac{F^e_*R}{F^e_*\fm}\right)
= [k^{1/p^e}:k].
\]
Therefore,
\[
a_e(R) \ge [k^{1/p^e}:k].
\]
Moreover, equality holds if and only if
$F^e_*I_e(R)=F^e_*\fm$, which is equivalent to $I_e(R)=\fm$.
Thus,
\begin{equation}\label{eq:splitting-ideal-equiv}
a_e(R) = [k^{1/p^e}:k] \quad \Longleftrightarrow \quad I_e(R) = \fm.
\tag{3.1.1}
\end{equation}
Now, by Equation~\hyperref[eq:splitting-ideal-equiv]{\ref*{eq:splitting-ideal-equiv}}
\[
\text{(i)} \Longleftrightarrow \text{(iv)}, \qquad
\text{(ii)} \Longleftrightarrow \text{(v)}, \qquad
\text{(iii)} \Longleftrightarrow \text{(vi)}.
\]
Clearly, (vi)$\Longrightarrow$(iv)$\Longrightarrow$(v).

\noindent(v)$\implies$ (iv): Since $R$ is Gorenstein, by \cite[Lemma~3.9]{Schwede2009}, there exist $\phi_{1}\in \Hom_{R}(F_{*}R,R)$ such that $\phi_e\coloneqq\phi_1\circ F_*^{e-1}\phi_{e-1}$ generates $\Hom_{R}(F^{e}_{*}R,R)$ as a $F^e_{*}R$-modules. Suppose that $I_e(R)=\fm$ for some $e$.
Then
\[
\fm=\phi_e(R)\subseteq \phi_1(R)\subseteq\fm,
\]
which implies $I_1(R)=\fm$.\\
\noindent (iv)$\implies$ (vi): Assume $I_1(R)=\fm$. We proceed by induction on $e$. The case $e=1$ is trivial.
Suppose $I_{e-1}(R)=\fm$. For any $x\in\fm$,
\[
\phi_e(F_*^ex)
= \phi_1(F_*\phi_{e-1}(F_*^{e-1}x))
\subseteq \phi_1(F_*\fm)
\subseteq \fm,
\]
so $I_e(R)=\fm$.\\
 \noindent (vi)$\implies$ (vii): If $\PP(R)\neq\fm$, then there exists $g\in\fm\setminus \PP(R)$. By definition of the splitting
prime, there exists $e$ such that the map
\[
R \xrightarrow{F_*^eg} F_*^eR
\]
splits, which implies $g\notin I_e(R)$, contradicting $I_e(R)=\fm$.\\
\noindent (vii)$\implies$ (iv): Since
\[
\PP(R)=\bigcap_{e\ge1} I_e(R)=\fm,
\]
it follows that $I_1(R)=\fm$.
\end{proof}
\begin{corollary}       
Let $(R,\mathfrak{m},k)$ be an $F$-finite, $F$-pure, Gorenstein  local ring of dimension $d$ of prime characteristic $p>0$ and $0\neq f\in \fm$. Suppose $(\underline{x})=(x_1, \dots, x_d)$ is a system of parameters so that $((\underline{x}) :_{R}\fm)=(x,u)$, where $u$ is a socle generator of $R/(x_1, \dots, x_d)$. Then the following are equivalent:
 \begin{enumerate}[label=(\roman*)]
     \item\label{cond:i} $a_1(R/(f)) = [k^{1/p} : k].$
     \item\label{cond:ii} There exists $e\in \mathbb{N}$ such that $a_e(R/(f))=[k^{1/p^e}:k]$.
     \item\label{cond:iii} For all $e\in \mathbb{N}$, $a_e(R/(f))=[k^{1/p^e}:k]$.  \item\label{cond:iv}$\left(((\underline{x})^{[p]}:_{R} u^p):_{R} f^{p-1}\right)=\fm$.
     \item\label{cond:v} There exists $e\in \mathbb{N}$ such that $\left(((\underline{x})^{[p^e]}:_{R} u^{p^e}):_{R} f^{p^{e}-1}\right)=\fm$.
     \item\label{cond:vi} For all $e\in \mathbb{N}$, $\left(((\underline{x})^{[p^e]}:_{R} u^{p^e}):_{R}f^{p^{e}-1}\right)=\fm$. 
     \item\label{cond:vii}  The splitting prime ideal $\mathcal{P}(R/(f))$ of $R/(f)$ equals $\fm/(f)$.
 \end{enumerate}
Moreover, if $(\underline{x})$ is a system of parameters of $R$, conditions (i)-(vii) holds true and if $\varepsilon\in(\underline{x})^{[p]}$, then $\mathcal{P}(R/(f+\varepsilon))=\mathfrak{m}/(f+\varepsilon)$.
\end{corollary}
\begin{proof}
By Theorem~\hyperref[thm:2.13]{\ref*{thm:2.13}} and Theorem~\hyperref[theorem:3.1]{\ref*{theorem:3.1}} conditions (i)-(vii) are equivalent.

Now assume that $(\underline{x}) = (x_1, \dots, x_d)$ is a system of parameters of $R$ and that conditions (i) through (vii) are satisfied. By (vii),
if $\PP(R/(f))=\fm/(f)$, then by Theorem~\hyperref[thm:2.13]{\ref*{thm:2.13}} and Proposition~\hyperref[proposition:2.15]{\ref*{proposition:2.15}},
\begin{align*}
 \mathcal{P}(R/(f))&=\bigcap\limits_{e\in \NN} I_{e}(R/(f))\\&=\frac{\Bigl((\underline{x})^{[p^{e}]}:_{R} u^{p^{e}}f^{p^{e}-1}\Bigr)}{(f)}\\&=\mathfrak {m}/(f).   
\end{align*}
Thus, $\Bigl((\underline{x})^{[p^{e}]}:_{R} u^{p^{e}}f^{p^{e}-1}\Bigr)=\fm$ for all $e\in \NN.$ Therefore $u^{p^{e}}f^{p^{e}-1}\in \Bigl((\underline{x})^{[p^{e}]}:\fm\Bigr)$.

Now, for $\varepsilon \in (\underline{x})^{[p]}$, we will show by induction on $e$ that
\[
u^{p^e} (f + \varepsilon)^{\,p^e-1} \in \bigl((\underline{x})^{[p^e]} : \mathfrak{m}\bigr).
\]
Since 
\[
(f + \varepsilon)^{p-1} = \sum_{i=0}^{p-1} \binom{p-1}{i} f^i \varepsilon^{\,p-1-i},
\] 
and for \(0 \le i < p-1\) we have \(p-1-i > 0\), it follows that 
\[
\varepsilon^{\,p-1-i} \in (x)^{[p]} \quad \text{for any } \varepsilon \in (x)^{[p]}.
\] 
Hence, 
\[
f^i \varepsilon^{\,p-1-i} \equiv 0 \pmod{(x)^{[p]}} \quad \text{for } 0 \le i < p-1,
\] 
and thus 
\[
(f + \varepsilon)^{p-1} \equiv f^{\,p-1} \pmod{(x)^{[p]}}.
\] 
Therefore, 
\[
u^p (f + \varepsilon)^{p-1} \in \bigl((x)^{[p]} : \fm\bigr).
\]
Now assume that the result is true for e, that is
\[
u^{p^{e}}(f+\varepsilon)^{p^{e}-1}\in \left((\underline{x})^{[p^e]}:\fm\right).\]
We will show that
\[
u^{p^{e+1}}(f+\varepsilon)^{p^{e+1}-1}\in \left((\underline{x})^{[p^{e+1}]}:\fm\right).\]
If $1 \le i \le p-1$, then 
\[
\varepsilon^i \in (\underline{x})^{[p]} \subseteq \mathfrak{m}^{[p]} = \mathcal{P}(R/(f))^{[p]}
\] 
and 
\begin{align*}
(u^{p^e} f^{\,p^e-1})^p& = u^{\,p^{\,e+1}} f^{\,p^{\,e+1}-p}\\& \in \bigl((\underline{x})^{[p^e]} : \mathfrak{m}\bigr)^{[p]}\\& \subseteq \bigl((\underline{x})^{[p^{\,e+1}]} : \mathfrak{m}^{[p]}\bigr).
\end{align*}

Therefore,
\begin{align*}
    \varepsilon^{i}u^{p^{e+1}}f^{p^{e+1}-1-i} &=\varepsilon^{i}u^{p^{e+1}}f^{p^{e+1}-p+p-1-i} \\&=\varepsilon^{i}u^{p^{e+1}}f^{p^{e+1}-p}f^{p-1-i}\\&\subseteq \varepsilon^{i}\Bigl((\underline{x})^{[p^{e+1}]} : \fm^{[p]}\Bigr) f^{p-1-i}\\&\subseteq (\underline{x})^{[p^{e+1}]}f^{p-1-i}\subseteq (\underline{x})^{[p^{e+1}]}.
\end{align*}
Suppose that $e_{0}\in \NN$ is such that $p^{e_0}\leq i\leq p^{e_0+1}-1$. If $i\geq p^{e_0}$, then $\varepsilon^i\in (\underline{x})^{[p^{e_0+1}]}\subseteq \fm^{[p^{e_0+1}]}$ since $\varepsilon\in (\underline{x})^{[p]}$.\\ Also,
\begin{align*}
   u^{p^{e+1}}f^{p^{e+1}-p^{e_{0}+1}}&=(u^{p^{e-e_0}}f^{p^{e-e_{0}}-1})^{[p^{e_{0}+1}]}\\&\in \Bigl((\underline{x})^{[p^{e-e_{0}}]} : \fm\Bigr)^{[p^{e_{0}+1}]}\\&=\Bigl((\underline{x})^{[p^{e+1}]} : \fm^{[p^{e_{0}+1}]}\Bigr). 
\end{align*}

Now, for $p^{e_0}\leq i\leq p^{e_0+1}-1$,
\begin{align*}
    \varepsilon^{i}u^{p^{e+1}}f^{p^{e+1}-1-i} &=\varepsilon^{i}u^{p^{e+1}}f^{p^{e+1}-p^{e_0+1}+p^{e_0+1}-1-i}\\&=\varepsilon^{i}u^{p^{e+1}}f^{p^{e+1}-p^{e_0+1}}f^{p^{e_0+1}-1-i} \\&=\varepsilon^{i}\bigl(u^{p^{e-e_0}}f^{p^{e-e_0}-1}\bigr)^{[p^{e_0+1}]}f^{p^{e_0+1}-1-i}\\&\subseteq \varepsilon^{i}\bigl((\underline{x})^{[p^{e-e_0}]} : \fm\bigr)^{[p^{e_0+1}]}f^{p^{e_0+1}-1-i}\\&=\varepsilon^{i}\bigl((\underline{x})^{[p^{e+1}]} : \fm^{[p^{e_0+1}]}\bigr)f^{p^{e_0+1}-1-i}\\&\subseteq (\underline{x})^{[p^{e+1}]}f^{p^{e_0+1}-1-i}\subseteq (\underline{x})^{[p^{e+1}]}.
\end{align*}
Hence, $u^{p^{e+1}}(f+\varepsilon)^{p^{e+1}-1}\equiv u^{p^{e+1}}f^{p^{e+1}-1}$ (mod $(\underline{x})^{[p^{e+1}]}$) for all $\varepsilon \in (\underline{x})^{[p]}$. Therefore, \[u^{p^{e+1}}(f+\varepsilon)^{p^{e+1}-1}\in \left((\underline{x})^{[p^{e+1}]}:\fm\right).\] Hence $u^{p^{e}}(f+\varepsilon)^{p^{e}-1}\in \left((\underline{x})^{[p^{e}]}:\fm\right)$, which implies $\PP\bigl(R/(f+\varepsilon)\bigr)=\fm/(f+\varepsilon).$ 
\end{proof}
Now, in this section, we investigate the $\fm$-adic continuity of the Frobenius splitting ratio of a divisor pair $(R,\Delta)$, where $R$ is a $\QQ$-Gorenstein, Cohen-Macaulay, $F$-finite local ring of prime characteristic $p>0$. A central question is how the Frobenius splitting ratio behaves under small perturbations of a nonzero divisor $f$. Theorems~\hyperref[theorem:3.3]{\ref*{theorem:3.3}} and~\hyperref[theorem:3.4]{\ref*{theorem:3.4}} provide key insights into this behavior. Theorem~\hyperref[theorem:3.3]{\ref*{theorem:3.3}} establishes conditions under which the Frobenius splitting number of $(R/(f),\Delta|_{f})$ remains unchanged after adding a small perturbation $\varepsilon$ to $f$. Specifically, the splitting prime ideal $(\PP(R/(f+\varepsilon)),\Delta|_{(f+\varepsilon)})$ is shown to be equal to $(\fm/(f+\varepsilon),\Delta|_{(f+\varepsilon)})$ under certain parameter conditions. Theorem~\hyperref[theorem:3.4]{\ref*{theorem:3.4}} further explores the comparison between the Frobenius splitting dimensions of $(R/(f),\Delta|_{f})$ and $(R/(f +\varepsilon),\Delta|_{(f+\varepsilon)})$, showing that under appropriate conditions, the Frobenius splitting dimension of $(R/\PP(R/(f)),\Delta|_{f})$ is less than or equal to that of $(R/\PP(R/(f + \varepsilon)),\Delta|_{(f+\varepsilon)})$. These results emphasize the improvement of the Frobenius splitting dimension under small $\fm$-adic perturbations.
\begin{theorem}\label{theorem:3.3}
Let $(R,\mathfrak{m},k)$ be a local $\QQ$-Gorenstein, Cohen--Macaulay,
$F$-finite, $F$-pure ring of prime characteristic $p>0$, and let $K_R$ be a canonical divisor on
$\operatorname{Spec}(R)$. Suppose that $\Delta \ge 0$ is an effective
$\QQ$-divisor such that $\Delta + K_R$ is $\QQ$-Cartier and, for all
sufficiently divisible $e \gg 0$, $(p^e-1)\Delta$ is integral.
Let $f \in R$ be a nonzero divisor such that the components of $\Delta$
are disjoint from the components of $\operatorname{div}(f)$,
$R/(f)$ satisfies $(S_2)$ and $(G_1)$, and
$\Delta|_{(f)}$ is a $\QQ$-divisor on $\operatorname{Spec}(R/(f))$.
Then there exist an $\mathfrak{m}$-primary ideal $\mathfrak{a} \subseteq R$
and an integer $e_0 \in \NN$ such that for all
$\varepsilon \in \mathfrak{a} \cap \PP^{[p^{e_0}]}$ and all $t \in \NN$,
\[
a_{t e_0}^{\Delta}(R/(f))
=
a_{t e_0}^{\Delta}(R/(f+\varepsilon)).
\]
In particular, if $\PP(R/(f),\Delta) = \mathfrak{m}/(f)$, then there exists
an $\mathfrak{m}$-primary ideal $\mathfrak{a}\subseteq R$ such that for all
$\varepsilon \in \mathfrak{a}$,
\[
\PP(R/(f+\varepsilon),\Delta) = \mathfrak{m}/(f+\varepsilon).
\]
\end{theorem}
\vspace{-\baselineskip}
\begin{proof}There exists an integer $e_0$ so that  $\Delta_{e_{0}}:=(p^{e_{0}}-1)(\Delta+K_R)$ is an integral divisor of index $p^{e_1}$. Let $(\underline{x})$ be a system of parameters, we will show that \[a_{te_{0}}^{\Delta}\bigl(R/(f)\bigr)=a_{te_{0}}^{\Delta}(R/(f+\varepsilon)),\] for any $\varepsilon\in (\PP(R/(f),\Delta))^{[p^{e_0}]}\cap (\underline{x})^{[p^{e_0}]}$, and $t\in \NN$.
First, we show that \[I_{te_0}^{\Delta}(R/(f))=I_{te_0}^{\Delta}(R/(f+\varepsilon)).\] The proof is based on induction on $t$. We begin by showing that $I_{e_0}^{\Delta}(R/(f))=I_{e_0}^{\Delta}(R/(f+\varepsilon))$. For this, let $g\in I_{e_0}^{\Delta}(R/(f)).$ By Proposition~\hyperref[proposition:2.15]{\ref*{proposition:2.15}}, we have \[gu^{p^{e_0}}f^{p^{e_0}-1}\in (\underline{x})^{[p^{e_0}]}R(\Delta_{e_0}+K_R).\] Let $\varepsilon\in (\PP(R/(f),\Delta))^{[p^{e_0}]}\cap (\underline{x})^{[p^{e_0}]}$, then \[gu^{p^{e_0}}f^{p^{e_0}-1}\in (\underline{x})^{[p^{e_0}]}R(\Delta_{e_0}+K_R)\] if and only if \[gu^{p^{e_0}}(f+\varepsilon)^{p^{e_0}-1}\in (\underline{x})^{[p^{e_0}]}R(\Delta_{e_0}+K_R).\]
Therefore, \[I_{te_0}^{\Delta}(R/(f))=I_{te_0}^{\Delta}(R/(f+\varepsilon)).\]\\
Assume that result is true for $t$, that is $I_{te_0}^{\Delta}\bigl(R/(f)\bigr)=I_{te_0}^{\Delta}\bigl(R/(f+\varepsilon)\bigr)$. We want to show that \[I_{(t+1)e_0}^{\Delta}\bigl(R/(f)\bigr)=I_{(t+1)e_0}^{\Delta}\bigl(R/(f+\varepsilon)\bigr).\] Let $g\in I_{(t+1)e_0}^{\Delta}(R/(f))$, by Proposition~\hyperref[proposition:2.15]{\ref*{proposition:2.15}}, which implies \[gu^{p^{(t+1)e_0}}f^{p^{(t+1)e_0}-1}\in (\underline{x})^{[p^{(t+1)e_0}]}R(\Delta_{(t+1)e_0}+K_R).\] 

If $1\leq i\leq p^{e_0}-1$, then $\varepsilon^{i}\in \bigl(\PP(R/(f),\Delta)\bigr)^{[p^{e_0}]}\cap (\underline{x})^{[p^{e_0}]}\subseteq (\underline{x})^{[p^{e_0}]}$ and 
\begin{align*}
    (u^{p^{te_0}}f^{p^{te_0}-1})^{p^{e_0}}&\in \left((\underline{x})^{[p^{te_0}]}R(\Delta_{te_0}+K_R):_{R(\Delta_{te_0}+K_R)}\PP(R/(f),\Delta)\right)^{[p^{e_0}]}\\& \subseteq (\underline{x})^{[p^{(t+1)e_0}]}R(\Delta_{(t+1)e_0}+K_R):_{R(\Delta_{(t+1)e_0}+K_R)}\PP(R/(f),\Delta)^{[p^{e_0}]}.
\end{align*}
Therefore,
\begin{align*}
    \varepsilon^{i}u^{p^{(t+1)e_0}}f^{p^{(t+1)e_0}-1-i} &=\varepsilon^{i}u^{p^{(t+1)e_0}}f^{{p^{(t+1)e_0}}-p^{e_0}+p^{e_0}-1-i} \\&=\varepsilon^{i}u^{p^{(t+1)e_0}}f^{{p^{(t+1)e_0}}-p^{e_0}}f^{p^{e_0}-1-i}\\&=\varepsilon^{i}\bigl(u^{p^{te_0}}f^{p^{te_0}-1}\bigr)^{p^{e_0}} f^{p^{e_0}-1-i} \\&\subseteq \varepsilon^{i} \left((\underline{x})^{[p^{(t+1)e_0}]}R(\Delta_{(t+1)e_0}+K_R):_{R(\Delta_{(t+1)e_0}+K_R)}\PP(R/(f),\Delta)^{[p^{e_0}]}\right)f^{p^{e_0}-1-i}\\&\subseteq (\underline{x})^{[p^{(t+1)e_0}]}R(\Delta_{(t+1)e_0}+K_R)f^{p^{e_0}-1-i}\\&\subseteq (\underline{x})^{[p^{(t+1)e_0}]}R(\Delta_{(t+1)e_0}+K_R).
\end{align*}
Suppose that $e_{1}\in \NN$ is such that $p^{e_1e_0}\leq i\leq p^{(e_1+1)e_0}-1$. If $i\geq p^{e_0e_1}$, then 
$\varepsilon^i\in \PP(R/(f),\Delta)^{[p^{(e_{1}+1)e_0}]},$ since $\varepsilon\in \PP(R/(f), \Delta)^{[p^{e_0}]}$. Also,
\begin{align*}
  u^{p^{(t+1)e_0}}f^{p^{(t+1)e_0}-p^{(e_{1}+1)e_0}}&=\Bigl(u^{p^{(t-e_1)e_0}}f^{p^{(t-e_{1})e_0}-1}\Bigr)^{[p^{(e_{1}+1)e_0}]}\\&\in \left((\underline{x})^{[p^{(t-e_1)e_0}]}R(\Delta_{(t-e_1)e_0}+K_R) : \PP(R/(f),\Delta)\right)^{[p^{(e_{1}+1)e_0}]}\\&\subseteq(\underline{x})^{[p^{(t+1)e_0}]}R(\Delta_{(t+1)e_0}+K_R) : \PP(R/(f),\Delta)^{[p^{(e_{1}+1)e_0}]}. 
\end{align*}
Now, for $p^{e_1e_0}\leq i\leq p^{(e_1+1)e_0}-1$, we have,
\begin{align*}
    \varepsilon^{i}u^{p^{(t+1)e_0}}f^{p^{(t+1)e_0}-1-i} &=\varepsilon^{i}u^{p^{(t+1)e_0}}f^{p^{(t+1)e_0}-p^{(e_1+1)e_0}+p^{(e_1+1)e_0}-1-i}\\&=\varepsilon^{i}u^{p^{(t+1)e_0}}f^{p^{(t+1)e_0}-p^{(e_1+1)e_0}}f^{p^{(e_1+1)e_0}-1-i} \\&=\varepsilon^{i}(u^{p^{(t-e_1)e_0}}f^{p^{(t-e_1)e_0}-1})^{[p^{(e_1+1)e_0}]}f^{p^{(e_1+1)e_0}-1-i}\\&\subseteq \varepsilon^{i}\left((\underline{x})^{[p^{(t-e_1)e_0}]}R(\Delta_{(t-e_1)e_0}+K_R) : \PP(R/(f),\Delta)\right)^{[p^{(e_{1}+1)e_0}]}f^{p^{(e_1+1)e_0}-1-i}\\&\subseteq\varepsilon^{i}\left((\underline{x})^{[p^{(t+1)e_0}]}R(\Delta_{(t+1)e_0}+K_R) : \PP(R/(f),\Delta)^{[p^{(e_{1}+1})e_0]}\right)f^{p^{(e_1+1)e_0}-1-i}\\&\subseteq (\underline{x})^{[p^{(t+1)e_0}]}R(\Delta_{(t+1)e_0}+K_R)f^{p^{(e_1+1)e_0}-1-i}\\& \subseteq (\underline{x})^{[p^{(t+1)e_0}]}R(\Delta_{(t+1)e_0}+K_R).
\end{align*}
Hence, \[gu^{p^{(t+1)e_0}}(f+\varepsilon)^{p^{(t+1)e_0}-1}\equiv gu^{p^{(t+1)e_0}}f^{p^{e+1}-1} (\text{mod} \ (\underline{x})^{[p^{(t+1)e_0}]}),\] for all $\varepsilon \in \left(\PP(R/(f),\Delta\right))^{[p^{e_0}]}\cap (\underline{x})^{[p^{e_0}]}$. Thus, the map
\[
    R/(f) \xrightarrow{\cdot F^{(t+1)e_0}_{*}g} F^{(t+1)e_0}_{*} R/(f)((p^{(t+1)e_0}-1)\Delta) \ \text { splits}
    \]
    if and only if the map
\[
    R/(f+\varepsilon) \xrightarrow{\cdot F^{(t+1)e_0}_{*}g} F^{(t+1)e_0}_{*} R/(f+\varepsilon) ((p^{(t+1)e_0}-1)\Delta)\ \text {splits.}
    \]
Therefore, $I_{(t+1)e_0}^{\Delta}(R/(f))=I_{(t+1)e_0}^{\Delta}(R/(f+\varepsilon))$. Hence, $a_{te_{0}}^{\Delta}(R/(f))=a_{te_{0}}^{\Delta}(R/(f+\varepsilon))$ if $\varepsilon\in (\PP(R/(f),\Delta))^{[p^{e_0}]}\cap (\underline{x})^{[p^{e_0}]}$.
 
If $(\PP(R/(f),\Delta)=(\fm/(f),\Delta)$, then by Theorem~\hyperref[thm:2.13]{\ref*{thm:2.13}} and Proposition~\hyperref[proposition:2.15]{\ref*{proposition:2.15}},
\begin{align*}
 \mathcal{P}(R/(f),\Delta)&=\bigcap\limits_{e_0\in \NN,t\geq 1} I^{\Delta}_{te_0}(R/(f))\\&=\frac{\Bigl((\underline{x})^{p^{te_0}}R(\Delta_{te_0}+K_R):_{R} u^{p^{te_0}}f^{p^{te_0}-1}\Bigr)}{(f)}\\&=\mathfrak {m}/(f).   
\end{align*}
Thus, $((\underline{x})^{p^{te_0}}R(\Delta_{te_0}+K_R):_{R} u^{p^{te_0}}f^{p^{te_0}-1})=\fm$. 

Now if $g\in \PP(R/(f),\Delta)$, then $g\in \fm$ so that $gu^{p^{te_0}}f^{p^{te_0}-1}\in (\underline{x})^{p^{te_0}}R(\Delta_{te_0}+K_R)$. Therefore $u^{p^{te_0}}f^{p^{te_0}-1}\in \left((\underline{x})^{p^{te_0}}R(\Delta_{te_0}+K_R):\fm\right)$.
But, 
\[\left((\underline{x})^{p^{te_0}}R(\Delta_{te_0}+K_R):\fm\right)=\left((\underline{x})^{p^{te_0}}R(\Delta_{te_0}+K_R), u(x_1x_2\dots x_d)^{p^{te_0}-1}\right),\] which implies $u^{p^{te_0}}f^{p^{te_0}-1}\in \left((\underline{x})^{p^{te_0}}R(\Delta_{te_0}+K_R), u(x_1x_2\dots x_d)^{p^{te_0}-1}\right)$. Now we can choose $\varepsilon\in \mathfrak{a}$ such that $\varepsilon\in \mathfrak{a}\subseteq (\underline{x})^{p^{te_0}}R(\Delta_{te_0}+K_R)$, that implies \[
u^{p^{te_0}}(f+\varepsilon)^{p^{te_0}-1}\in \left((\underline{x})^{p^{te_0}}R(\Delta_{te_0}+K_R), u(x_1x_2\dots x_d)^{p^{te_0}-1}\right).\]
Thus, $u^{p^{te_0}}(f+\varepsilon)^{p^{te_0}-1}\in \Bigl((\underline{x})^{p^{te_0}}R(\Delta_{te_0}+K_R):_{R}\fm\Bigr)$, hence $\PP(R/(f+\varepsilon),\Delta)=\fm/(f+\varepsilon).$  
\end{proof}

Let $(R,\fm,k)$ be a $\QQ$-Gorenstein, Cohen-Macaulay, $F$-finite local ring of dimension $d+1$ of prime characteristic $p>0$ and $f\in \fm$ be a nonzero divisor of $R$. Theorem~\hyperref[theorem:3.4]{\ref*{theorem:3.4}} explores the relation between the splitting dimensions of $(R/(f+\varepsilon),\Delta|_{(f+\varepsilon)})$ and $(R/(f),\Delta|_{f})$ for $\varepsilon\in \fm^{[p^{e_0}]}$. The result demonstrates the existence of a natural number $e_0\in \NN$ such that for certain perturbations $\varepsilon\in \fm^{[p^{e_0}]}$, the Frobenius splitting dimension of $(R/\PP(R/(f),\Delta))$ is less than or equal to the Frobenius splitting dimension of $(R/\PP(R/(f + \varepsilon),\Delta))$. The following is the second main result of this section.

\begin{theorem}\label{theorem:3.4}
  Let $(R,\mathfrak{m},k)$ be a local \ $\mathbb{Q}$-Gorenstein, Cohen-Macaulay, $F$-finite, $F$-pure ring of dimension $(d+1)$ and of prime characteristic $p>0$, and let $K_R$ be a canonical divisor of $\operatorname{Spec}(R)$. Suppose that $\Delta \ge 0$ is an effective $\mathbb{Q}$-divisor such that $\Delta + K_R$ is $\mathbb{Q}$-Cartier and that for all $e \gg 0$ and divisible, $(p^e - 1)\Delta$ is integral. Let $f \in R$  be a nonzero divisor such that the components of $\Delta$ are disjoint from the components of $\operatorname{div}(f)$, $R/(f)$ satisfies $(S_2)$ and $(G_1)$, and $\Delta|_{(f)}$ is a $\mathbb{Q}$-divisor on $\operatorname{Spec}(R/(f))$. Then there exists $e_0 \in \NN$ such that for all $\varepsilon \in \mathfrak{m}^{[p^{e_0}]}$,
\[
\dim \Bigl(R/\PP(R/(f),\Delta)\Bigr)
\;\le\;
\dim \Bigl(R/\PP(R/(f+\varepsilon),\Delta)\Bigr).
\]  
\end{theorem}
\begin{proof} Assume that $\dim (R/\PP(R/(f),\Delta))=h$, then $\text{ht}(\PP(R/(f),\Delta))=d-h$, since $\dim ((R/(f),\Delta))=d$. Since $\dim (R/(\PP(R/(f),\Delta)))=h$, there exist parameters $x_1,x_2,\dots,x_h$ of $R/(f)$ and for each $1\leq i\leq h,$
  \[
    R/(f) \xrightarrow{\cdot F^{te_{i}}_{*}x_i} F^{te_{i}}_{*} R/(f)((p^{te_i}-1)\Delta) \ \text {splits.}
    \]
    
Complete $f,x_1,x_2,\dots,x_h$ to a full system of parameters $(\underline{x})=f,x_1,x_2,\dots,x_d$. Let $u$ generate the socle $((\underline{x}):_{R}\fm)/(\underline{x})$. Then by Proposition~\hyperref[proposition:2.15]{\ref*{proposition:2.15}} \[x_iu^{p^{te_i}}f^{p^{t e_i}-1}
\notin (\underline{x})^{[p^{te_i}]}R(\Delta_{te_i}+K_R),\] for all $1\leq i \leq h$, which is true if and only if \[u^{p^{te_i}}f^{p^{t e_i}-1}
\notin \left((\underline{x})^{[p^{te_i}]}R(\Delta_{te_i}+K_R):x_i\right).\]

Now, by \cite{SrinivasTrivedi1996}, if 
$f, x_1, x_2, \dots, x_h, x_{h+1}, \dots, x_d$ is a system of parameters of $R$, then there exists $n \in \NN$ such that for all 
$\varepsilon_0, \varepsilon_1, \dots, \varepsilon_d \in \fm^n$, the elements 
$f+\varepsilon_0, x_1+\varepsilon_1, \dots, x_d+\varepsilon_d$ form a system of parameters of $R$. The elements $x_1, x_2, \dots, x_d$ are parameters for $R/(f)$ if and only if $f, x_1, x_2, \dots, x_d$ are parameters for $R$. Therefore, 
$x_1+\varepsilon_1, x_2+\varepsilon_2, \dots, x_d+\varepsilon_d$ are parameters of $R$, which implies that 
$x_1, x_2, \dots, x_d$ are parameters for $R/(f+\varepsilon)$.

Now, we want to show that if $\varepsilon \in \fm^{n \gg 0}$ and $1 \leq i \leq h$, then
\[
    R/(f+\varepsilon) \xrightarrow{\cdot F^{te_i}_{*}x_i} F^{te_i}_{*} R/(f+\varepsilon)((p^{te_i}-1)\Delta) \quad \text{splits.}
\]

That is,
\[
    u^{p^{te_i}}(f+\varepsilon)^{p^{te_i}-1} \notin \big((\underline{x})^{[p^{te_i}]}R(\Delta_{te_i}+K_R) : x_i\big).
\]

Let $e_0 = \max\{e_1, e_2, \dots, e_h\}$. We can choose $\varepsilon \in \fm^n$ such that $\fm^n \subseteq \fm^{[p^{e_0}]}$ and $f+\varepsilon$ is regular. Since $\varepsilon \in \fm^{[p^{e_0}]}$ and 
\[
    u^{p^{te_0}} f^{p^{te_0}-1} \notin \big((\underline{x})^{[p^{te_0}]}R(\Delta_{te_0}+K_R) : x_i\big),
\] 
it follows that 
\[
    u^{p^{te_0}} (f+\varepsilon)^{p^{te_0}-1} \notin \big((\underline{x})^{[p^{te_0}]}R(\Delta_{te_0}+K_R) : x_i\big).
\]

Thus, $x_1, \dots, x_h \notin \PP(R/(f+\varepsilon), \Delta)$, and hence 
$ \text{ht}\bigl(\PP(R/(f+\varepsilon),\Delta)\bigr) \geq d-h,$
which implies
\[
    \dim\Bigl(R/\PP(R/(f),\Delta)\Bigr) \leq \dim\Bigl(R/\PP(R/(f+\varepsilon),\Delta)\Bigr).
\]

\end{proof}


\section{Examples}\label{sec:example}
In this section, we illustrate the theoretical results with specific examples. Example~\hyperref[example:4.1]{\ref*{example:4.1}} examines a regular local ring of characteristic $7$, where a hypersurface defined by 
\[
f = x_0^2 - x_1^6 x_2^2 + x_3^3
\] 
is first shown to be $F$-pure. We then prove that $R/(f+\varepsilon)$ is $F$-pure as well for $\varepsilon \in \PP(R/(f))^{[p]}$. Finally, by calculating the Jacobian ideal, we show that the splitting dimensions of $R/(f)$ and $R/(f+\varepsilon)$ lift to the same prime ideal of $R$, although $R/(f) \not\cong R/(f+\varepsilon)$.
\begin{example}\label{example:4.1}
Let $R=k[[x_0,x_1,x_2,x_3]]$ be a ring of characteristic $p=7$ with $k=\bar{k}$, and let 
\[
f = x_0^{2} - x_1^{6} x_2^{2} + x_3^{3} \in R.
\] 
According to Fedder's Criterion \cite{Fedder83}, $R/(f)$ is $F$-pure. Indeed, we have $f^{7-1} = (x_0^2 - x_1^6 x_2^2 + x_3^3)^6$, which includes the term $(x_0^2)^3 (x_1^6 x_2^2) (x_3^3)^2$, and this term is not in $\fm^{[7]} = (x_0^7,x_1^7,x_2^7,x_3^7)$. Moreover, 
\[
x_2 (x_0^2)^3 (x_1^6 x_2^2) (x_3^3)^2 \notin \fm^{[7]},
\] 
hence $x_2 \notin \PP(R/(f))$. 

Consider the radical of the Jacobian ideal of $f$:
\[
\sqrt{Jac(f)} = \sqrt{(2x_0, -6x_1^5 x_2^2, -2x_1^6 x_2, 3x_3^2)}.
\] 
Setting $2x_0=0$ gives $x_0=0$. The equation $-6x_1^5 x_2^2=0$ implies $x_1=0$ or $x_2=0$, and $-2x_1^6 x_2=0$ similarly implies $x_1=0$ or $x_2=0$. Finally, $3x_3^2=0$ implies $x_3=0$. Therefore, the singular locus of $R$ is the set of points where $x_0=0$, $x_3=0$, and either $x_1=0$ or $x_2=0$. Equivalently, the varieties $V(x_0,x_1,x_3)$ and $V(x_0,x_2,x_3)$ consist of points of the form $(0,0,x_2,0)$ and $(0,x_1,0,0)$, respectively. Their intersection, $V(x_0,x_1,x_3)\cap V(x_0,x_2,x_3)$, corresponds to the points where $x_0=0$, $x_3=0$, and either $x_1=0$ or $x_2=0$. Thus, the singular locus is 
\[
\text{Sing}(R) = V(x_0,x_1,x_3) \cap V(x_0,x_2,x_3).
\] 
In particular, by \cite[Proposition~3.6]{AberbachEnescu2005}, for any prime ideal $q \in \text{Spec}(R)$ with $\PP(R)\subseteq q$, we have $\PP(R)R_q = \PP(R_q)$. Hence, the splitting prime ideal of $R$ is not the maximal ideal and must be either $(x_0,x_1,x_3)$ or $(x_0,x_2,x_3)$. We show that $(R/(f))_{(x_0,x_2,x_3)}$ is strongly $F$-regular.

Observe that 
\[
x_0 ((x_0^2)^2 (x_1^6 x_2^2)^3 (x_3^3)^2) \notin (x_0^7, x_2^7, x_3^7) R_{(x_0,x_2,x_3)}.
\] 
The term $(x_0^2)^2 (x_1^6 x_2^2)^3 (x_3^3)^2$ appears in the expansion of $f^6$. By \cite{Glassbrenner1998}, $(R/(f))_{(x_0,x_2,x_3)}$ is strongly $F$-regular since $(R/(f))_{x_0}$ is regular. Therefore, the splitting prime ideal of $(R/(f))_{(x_0,x_2,x_3)}$ is zero. Hence $R/(f)$ is $F$-pure, and 
\[
\PP(R/(f)) = (x_0,x_1,x_3).
\]

Let $\varepsilon = x_1^n$. Then $(f,x_1,x_2,x_3)$ is a system of parameters of $R$. By Theorem~\hyperref[theorem:3.3]{\ref*{theorem:3.3}}, if $\varepsilon \in \PP(R/(f))^{[p]} \cap (f,x_1,x_2,x_3)^{[p]}$, then $a_e(R/(f)) = a_e(R/(f+\varepsilon))$ for all $e \in \NN$. Hence 
\[
a_e(R/(f)) = a_e(R/(f+x_1^n)) \quad \text{for all } n \geq 7.
\] 
If $p \nmid n$, then 
\[
\sqrt{Jac(f+x_1^n)} = (x_0,x_1,x_3) = \PP(R/(f)).
\] 
Therefore, $R/(f) \not\cong R/(f+x_1^n)$ for all $n>7$ with $7 \nmid n$, while 
\[
a_e(R/(f)) = a_e(R/(f+x_1^n)) \quad \text{for all } e \in \NN.
\]
\end{example}

  But, equality of Frobenius splitting dimension under small perturbation does not hold in general, even if $R$ is non-singular. 
\begin{example}\label{ex:four-two}
Let $k=\bar{k}$ be a field of characteristic $7$, and let 
\[
R = k[x,y,z,w]_{(x,y,z,w)}, \quad f = x^3 + y^3 + z^3 \in R.
\] 
By Fedder's criterion \cite{Fedder83}, $R/(f)$ is $F$-pure. Indeed, 
\[
\sqrt{Jac(f)} = \sqrt{(3x^2, 3y^2, 3z^2)}.
\] 
Hence, the singular locus of $R/(f)$ is $V(x,y,z)$.  

The elements $F^e_{*} 1, F^e_{*} w, \dots, F^e_{*} w^{p^e-1}$ generate free summands of $F^e_{*} R$, since 
\[
w^i (x^3 + y^3 + z^3)^{p^e-1} \notin \fm^{[p^e]} \quad \text{for all } 0 \leq i \leq p^e - 1.
\] 
Therefore, the splitting prime ideal of $R$ is 
\[
\PP(R/(f)) = (x,y,z),
\] 
and we have 
\[
a_e(R/(f)) = p^e \quad \text{and} \quad r_F(R/(f)) = 1.
\] 

Let $n \not\equiv 0 \pmod{7}$ and $\varepsilon = w^n$. We claim that $R/(f+\varepsilon)$ is strongly $F$-regular.  

Since $7 \nmid n$, the hypersurface $R/(f+\varepsilon)$ has an isolated singularity at $(x,y,z,w)$ by the Jacobian criterion. In $R/(f+\varepsilon)$, the localization $(R/(f+\varepsilon))_w$ is regular, and 
\[
w \big( (x^3)^2 (y^3)^2 (z^3)^2 \big) \notin (x^7, y^7, z^7, w^7) = \fm^{[7]},
\] 
where $(x^3)^2 (y^3)^2 (z^3)^2$ is a term in 
\[
(f+\varepsilon)^{p-1} = (f+\varepsilon)^6 = (x^3 + y^3 + z^3 + w^n)^6.
\] 
Hence, by \cite{Glassbrenner1998}, $R/(f+\varepsilon)$ is strongly $F$-regular.  

Since $R/(f+\varepsilon)$ is strongly $F$-regular, we have $\PP(R/(f+\varepsilon)) = 0$, which implies 
\[
\dim(R/\PP(R/(f+\varepsilon))) = \dim(R/(f+\varepsilon)) = 3 > \dim(R/\PP(R/(f))) = 1.
\] 
Moreover, $R/(f+\varepsilon)$ is singular, so by \cite[Corollary~16]{Huneke_2002}, we have 
\[
r_F(R/(f+\varepsilon)) = s(R/(f+\varepsilon)) < 1 = r_F(R/(f)).
\] 
\end{example}

\section{References}
    \nocite{*}
    \begin{center}\mbox{}
    \vspace{-\baselineskip}
    \printbibliography[heading=none]
    \end{center}

\end{document}